\documentclass[12pt]{amsart}

\usepackage{amsmath}
\usepackage{amsfonts}
\usepackage{fullpage}

\newtheorem{theorem}{Theorem}[section]
\newtheorem{lemma}[theorem]{Lemma}
\newtheorem{proposition}[theorem]{Proposition} 
\newtheorem{corollary}[theorem]{Corollary} 

\theoremstyle{definition}
\newtheorem{definition}[theorem]{Definition}

\theoremstyle{remark}
\newtheorem{remark}[theorem]{Remark}

\numberwithin{equation}{section}

\newcommand{\disk}{\ensuremath{\mathbb{D}} } 
\newcommand{\cdisk}{\ensuremath{\overline{\mathbb{D}}_0}} 

\newcommand{\teich}{Teichm\"uller } 

\newcommand{\id}{\operatorname{id}} 

\newcommand{\riem}{\Sigma}

\newcommand{\tauP}{\tilde{\tau}}

\newcommand{\pmcgi}[1][]{\operatorname{PModI}({#1})} 
\newcommand{\DB}{\operatorname{DB}} 
\newcommand{\PDB}{\mathcal{P}} 

\renewcommand{\Bbb}[1]{\ensuremath{\mathbb{#1}}}
\newcommand{\st}{\, | \,} 

\newcommand{\C}{\mathcal{C}} 
\newcommand{\Oqc}{\mathcal{O}_{\mathrm{qc}}} 

\begin{document}

\title[Fiber Structure and Local Coordinates for Teichm\"uller Space]{Fiber Structure and Local Coordinates for the Teichm\"uller Space of a Bordered Riemann Surface}

\author{David Radnell}
\address{Department of Mathematics and Statistics \\
American University of Sharjah \\
PO BOX 26666, Sharjah, UAE} 
\email[D. ~Radnell]{dradnell@aus.edu}

\author{Eric Schippers}
\address{Department of Mathematics \\
University of Manitoba\\
Winnipeg, MB, R3T 2N2, Canada} 
\email[E.
~Schippers]{eric\_schippers@umanitoba.ca}


\date{\today}


\begin{abstract}
We show that the infinite-dimensional Teichm\"uller space of a
Riemann surface whose boundary consists of $n$ closed curves is a
holomorphic fiber space over the Teichm\"uller space of
$n$-punctured surfaces. Each fiber is a complex Banach manifold
modeled on a two-dimensional extension of the universal
Teichm\"uller space. The local model of the fiber, together with the
coordinates from internal Schiffer variation, provides new
holomorphic local coordinates for the infinite-dimensional
Teichm\"uller space.
\end{abstract}

\maketitle

\begin{section}{Introduction}
\label{introduction}

\begin{subsection}{Statement of results}

Let $\mathbb{D} = \{ z \in \mathbb{C} \,:\, |z|<1\}$, $\mathbb{D}_0 = \mathbb{D} \setminus \{0\}$, and $\cdisk = \{ z \in \mathbb{C} \,:\, 0<|z|\leq 1\}$.
\begin{definition}  We say that $\riem^B$ is a bordered Riemann
 surface of type $(g,n)$ if 1) its boundary consists of $n$ ordered closed
 curves homeomorphic to $S^1$ and 2) it is biholomorphically equivalent to
 a compact Riemann surface of genus $g$ with $n$ simply-connected
 non-overlapping regions, biholomorphic to $\mathbb{D}$,
 removed.  We say that $\riem^P$ is a punctured Riemann surface of type
 $(g,n)$ if it is biholomorphically equivalent to a compact Riemann
 surface with $n$ points $p_1,\ldots,p_n$ removed.
\end{definition}
Note that it is assumed that bordered Riemann surfaces have no
punctures, and that the boundary components of punctured Riemann surfaces
consist only of single points.   We will denote bordered Riemann
surfaces of type $(g,n)$ with a superscript $B$ and punctured
Riemann surfaces with a superscript $P$.
\begin{remark}
 One can also view a punctured Riemann surface $\riem^P$ as a
 compact Riemann surface with $n$ distinguished points.
\end{remark}

\begin{remark}  Throughout the paper, we consider only Riemann surfaces with no non-trivial
 automorphisms that are homotopic to the identity. That is, we do not consider the special cases where $2g-2+n \leq 0$.
\end{remark}

Given a bordered Riemann surface $\riem^B$ of type $(g,n)$, we can
obtain a punctured Riemann surface $\riem^P$ in the following way.
For details see \cite[Section 3]{RS05}. Denote the boundary
components by $\partial_i \riem^B$.  Let
$\tau=(\tau_1,\ldots,\tau_n)$ where each $\tau_i:\partial \mathbb{D}
\rightarrow
\partial_i \riem^B$ is a fixed quasisymmetric mapping. For the
purposes of this paper, we say that such a $\tau_i$ is
quasisymmetric if $\tau_i$ extends to a quasiconformal map of $\{z:
1<|z|<r\}$ into a doubly-connected neighborhood of $\partial_i
\riem^B$  \cite[Definition 2.12]{RS05}. We sew on $n$ copies of the
punctured unit disc $\cdisk$, denoted $\overline{\mathbb{D}}_{0,i}$,  to $\riem^B$ as follows. Consider the
disjoint union of $\riem^B$ and $\overline{\mathbb{D}}_{0,1}
\sqcup \cdots \sqcup \overline{\mathbb{D}}_{0,n}$.  Identifying boundary points using  $\tau$, the
result is a compact surface $\riem^P$ with $n$ punctures
$p_i$ corresponding to each puncture in $\cdisk$. That is, let
\[  \riem^P = (\riem^B \sqcup \overline{\mathbb{D}}_{0,1}
\cdots \sqcup \overline{\mathbb{D}}_{0,n})/\sim \]
where $p \sim q$ if $p \in \partial_i \riem^B$ and $q \in
\partial \mathbb{D}_i$, and $p=\tau_i(q)$.  By \cite[Theorem 3.3]{RS05} $\riem^P$
has a unique complex structure which is compatible with that of
both $\riem^B$ and $\overline{\mathbb{D}}_{0,i}$ for all $i$. (Note that the
punctures of $\riem^P$ are also ordered.)  If $\riem^P$ is
obtained from $\riem^B$ in this way we will say that $\riem^P$ is
obtained by ``sewing caps on $\riem^B$ via $\tau$'' and we write
$$
\riem^P = \riem^B \#_{\tau} \cdisk^n .
$$
The parametrizations $\tau_i$ can be extended to maps $\tilde{\tau}_i: \cdisk \to \riem^P$ to the caps of $\riem^P$ by
\begin{equation}
\label{eq:tauextension_definition}
\tilde{\tau}_i(x) =
\begin{cases}
\tau_i(x) & \text{for }  x \in \partial \disk \\
x & \text{for } x \in \disk_0 .
\end{cases}
\end{equation}
These maps $\tauP_i$ have quasiconformal extensions to a neighborhood of $\cdisk$.

The aim of this paper is to show that the infinite-dimensional Teichm\"uller space
$T(\riem^B)$ of $\riem^B$ is a holomorphic fiber space over the
finite-dimensional Teichm\"uller space $T(\riem^P)$ of $\riem^P$. The fibers can be
explicitly described as follows.
\begin{definition}[of $\mathcal{O}_{qc}({\riem}_*^P)$] \label{de:Oqcdefinition}
 Let $\Sigma_*^P$ be a compact Riemann surface of type $(g,n)$, with distinguished points
 $p_1,\ldots,p_n$.
 Define $\Oqc(\riem_*^P)$ to be the set of $n$-tuples
 $(\phi_1,\ldots,\phi_n)$
 where $\phi_i:\mathbb{D} \rightarrow \Sigma_*^P$ are maps
 with the following
 properties:
 \begin{enumerate}
  \item $\phi_i(0)=p_i$
  \item $\phi_i$ is conformal on $\mathbb{D}$
  \item $\phi_i$ has a quasiconformal extension to a neighborhood of
   $\overline{\mathbb{D}}$
  \item $\phi_i(\overline{\mathbb{D}}) \cap
  \phi_j(\overline{\mathbb{D}}) = \emptyset$ whenever $i \neq j$.
 \end{enumerate}
\end{definition}
It is convenient to single out the following case.
\begin{definition}
 Let $\Oqc$ denote the set of holomorphic univalent functions $f:\mathbb{D}
 \rightarrow {\mathbb{C}}$
 with quasiconformal extensions to ${\mathbb{C}}$ satisfying
 the normalization $f(0)=0$.
\end{definition}

In \cite{RSOqc} it was shown that $\Oqc$ possesses a natural complex
structure related to that of the universal Teichm\"uller space, and
that $\Oqc(\riem_*^P)$ is a complex Banach manifold which is locally
modeled on $\Oqc^n= \Oqc \times \cdots \times \Oqc$.

Let $\riem^P$ be the punctured Riemann surface obtained from $\riem^B$
by sewing on caps via $\tau$, and let $\riem_*^P$ be a marked
Riemann surface representing an element of the Teichm\"uller space
$T(\riem^P)$ of $\riem^P$. We show that  the fiber in $T(\riem^B)$
over this element of $T(\riem^P)$, modulo the action of Dehn twists
around curves homotopic to boundary curves, is in a natural one-to-one correspondence
with $\Oqc(\riem_*^P)$.

We summarize the main results.  Precise definitions and careful
statements of the theorems can be found in subsequent sections.

\smallskip
\noindent \textbf{Summary of results}
\begin{enumerate}
 \item If $\riem^B$ is a bordered Riemann surface of type $(g,n)$ (with $2g-2+n > 0$)
 then the Teichm\"uller space $T(\riem^B)$ is a complex fiber space
 over $T(\riem^P)$.
 \item Let $F_B([\riem^P,f,\riem_*^P])$ be the fiber over
 $[\riem^P,f,\riem_*^P] \in T(\riem^P)$.
  Let $\DB$ be the subgroup of the mapping class
  group of $\riem^B$ corresponding
  to Dehn twists around curves homotopic to boundary curves.  Then $F_B([\riem^P,f,\riem_*^P])/\DB$ is
  biholomorphic to $\Oqc(\riem_*^P)$.  In particular
  $F_B([\riem^P,f,\riem_*^P])$ is locally biholomorphic to the function space
  $\Oqc^n$.
 \item Schiffer variation coordinates on $T(\riem^P)$ together with open subsets of $\Oqc(\riem_*^P)$ give local holomorphic coordinates for $T(\riem^B)$.
\end{enumerate}

Some recent results are a key part of the formulation and proofs of
these theorems.  First is the authors' construction in \cite{RS05}
of a complex structure on the ``rigged moduli space'' and its
explicit relation to the Teichm\"uller space $T(\riem^B)$.  This
construction uses ideas from two-dimensional conformal field theory
in an essential way.  A further crucial result from \cite{RS05} is
that the operation of sewing two Riemann surfaces with a
quasisymmetric boundary identification is holomorphic in
Teichm\"uller space.  The complex structure on $\Oqc(\riem_*^P)$ was
constructed in \cite{RSOqc}.

A further tool is Gardiner's construction of coordinates on the
Teichm\"uller space of a surface of finite type using Schiffer
variation \cite{Gardiner}.  We use his method in order to construct
a section of the projection of $T(\riem^B)$ onto $T(\riem^P)$. In
the first author's thesis \cite{RadThesis}, Schiffer variation was
used in an analogous way, although for the different purpose of
defining a complex structure on the analytically rigged
Teichm\"uller space.

The results above are an application of Segal's \cite{Segal} formulation of conformal field theory to \teich theory. 
\end{subsection}

\end{section}

\begin{section}{fiber structure of $T(\riem^B)$}

\begin{subsection}{Teichm\"uller spaces}
\label{TeichandMCG}

 We now define the relevant Teichm\"uller spaces, to fix notation.

 Let $\riem^B$ be a bordered Riemann surface of of type $(g,n)$.
 Consider the set of triples
 $\{(\riem^B,f_1,\riem_{1}^B)\}$ where $\riem^B$ is a fixed Riemann
 surface, $\riem_{1}^B$ is another Riemann surface and $f_1:\riem^B
 \rightarrow \riem_{1}^B$ is a quasiconformal map (the ``marking'').
 We say that $(\riem^B,f_1,\riem_{1}^B) \sim (\riem^B,f_2,\riem_{2}^B)$
 if there exists a biholomorphism $\sigma:\riem_{1}^B \rightarrow
 \riem_{2}^B$ such that $f_2^{-1} \circ \sigma \circ f_1$ is
 homotopic to the identity ``rel boundary'', i.e. in such a way that
 the restriction of $f_2^{-1} \circ \sigma \circ f_1$ to the boundary
 is the identity throughout the homotopy.
 \begin{definition}
 The Teichm\"uller space of $\riem^B$ is
 \[  T(\riem^B) = \{ (\riem^B,f_1,\riem_{1}^B)\}/\sim.  \]
 \end{definition}
 We denote the equivalence classes by $[\riem^B,f_1,\riem_{1}^B]$.
 The case $T(\mathbb{D}^*)$, where $\mathbb{D}^* = \{z \,:\,
 |z|>1 \} \cup \{\infty\}$, is the universal Teichm\"uller space.

 It is well known that $T(\riem^B)$ is a complex Banach manifold with complex structure
 compatible with the space $L^\infty_{-1,1}(\riem^B)$ of Beltrami differentials
 $\mu$ on $\riem^B$ in the following sense.  Let
 $L^\infty_{-1,1}(\riem^B)_1$ denote the unit ball in the space of
 Beltrami differentials.  Let
 \[  \Phi:L^\infty_{-1,1}(\riem^B)_1 \rightarrow T(\riem^B) \]
 be the fundamental projection, given by taking a Beltrami
 differential $\mu$ to $[\riem^B,f^\mu,\riem^{\mu}]$ where $f^\mu$
 is a quasiconformal map of dilatation $\mu$.  $\Phi$ is well-defined
 and
 \begin{theorem}[\cite{Lehto},\cite{Nag}] \label{th:Phiinducescomplexstructure}
  $\Phi$ is holomorphic and possesses local
  holomorphic sections.
 \end{theorem}

 Similarly, we can define the Teichm\"uller space of a punctured
 Riemann surface.  Let $\riem^P$ be a punctured Riemann surface of
 type $(g,n)$.
 Consider the set of triples
 $(\riem^P,f_1,\riem_{1}^P)$ where $\riem^P$ is a fixed punctured Riemann
 surface, and $f_1: \riem^P \rightarrow \riem_{1}^P$ is a
 quasiconformal map onto the Riemann surface $\riem_{1}^P$.  We say
 that $(\riem^P,f_1,\riem_{1}^P) \sim (\riem^P,f_2,\riem_{2}^P)$ if
 there is a biholomorphism $\sigma:\riem_{1}^P \rightarrow
 \riem_{2}^P$ such that $f_2^{-1} \circ \sigma \circ f_1$ is
 homotopic  to the identity.
 \begin{definition}
  The Teichm\"uller space of $\riem^P$
 is  defined by
 \[ T(\riem^P) = \{ (\riem^P,f_1,\riem_{1}^P)\} /\sim.  \]
 \end{definition}
 We denote the equivalence classes by $[\riem^P,f_1,\riem_{1}^P]$.
 The Teichm\"uller space $T(\riem^P)$ is a complex manifold of dimension $3g-3 +n$.
 \begin{remark}
  Observing that quasiconformal and holomorphic maps must extend to
  the punctures, one obtains an alternate description of punctured
  Riemann surfaces and their Teichm\"uller spaces.
  One regards a punctured Riemann
  surface as a compact surface with distinguished points, and the
  definition of $T(\riem^P)$ is altered to require that the
  quasiconformal mappings take punctures to punctures, and the
  homotopy is ``rel punctures'', that is it preserves the punctures
  throughout.
 \end{remark}
 \begin{remark}
  Of course, the Teichm\"uller spaces of punctured
  and bordered Riemann surfaces are special cases of the
  general definition of the Teichm\"uller space of any Riemann surface
  covered by the disc.
 \end{remark}

As in \cite[section 2.1]{RS05} we introduce a certain subgroup of
the mapping class group. The pure mapping class group of $\riem^B$
is the group of homotopy classes of quasiconformal self-mappings of
$\riem^B$ which preserve the ordering of the boundary components.
Let $\pmcgi[\riem^B]$ be the subgroup of the mapping class group
consisting of equivalence classes of mappings that are the identity
on the boundary  $\partial \riem^B$. The group $\pmcgi[\riem^B]$ is
finitely generated by Dehn twists.
\begin{definition}
Let $\DB(\riem^B)$ be the subgroup of $\pmcgi[\riem^B]$ generated by the equivalence classes of mappings which are Dehn twists around curves that are homotopic to boundary curves. We write $\DB$ when the surface is clear from context.
\end{definition}
If $\riem^B$ is not an annulus or a disk then $\DB(\riem^B)$ is isomorphic to $\mathbb{Z}^n$ and is in the center of  $\pmcgi[\riem^B]$.

The mapping class group acts on $T(\riem^B)$ by $[\rho] \cdot [\riem^B,f,\riem_1^B] = [\riem^B, f \circ \rho, \riem_1^B]$.
\begin{proposition}[{\cite[Lemmas 5.1 and 5.2]{RS05}}]
\label{DBaction}
The group $\pmcgi[\riem^B]$, and hence its subgroup $\operatorname{DB}(\riem^B)$, acts properly discontinuously and fixed-point freely by biholomorphisms on $T(\riem^B)$.
\end{proposition}

\end{subsection}

\begin{subsection}{Description of the fibration}
\label{se:fibration}

 In this section we describe the fibers.
This requires the use of a ``rigged
Teichm\"uller space'', a concept motivated by conformal field
theory, which we will define in the next paragraph.  Details and proofs of the
statements in
this section were given in \cite{RS05}.
\begin{definition}
 Let $\Sigma^P$ be a type $(g,n)$ punctured Riemann surface.  Consider the set of
 quadruples $(\riem^P,f_1,\riem_{1}^P,\phi)$ where $\riem_{1}^P$ is a
 Riemann surface, $f_1:\riem^P \rightarrow \riem_{1}^P$ is a
 quasiconformal marking map as in the definition of $T(\riem^P)$,
 and $\phi \in \Oqc(\riem_{1}^P)$. Define the equivalence relation
 $(\riem^P,f_1,\riem_{1}^P,\phi_1) \sim
 (\riem^P,f_2,\riem_{2}^P,\phi_2)$ if and only if there exists a
 biholomorphism $\sigma:\riem_{1}^P \rightarrow \riem_{2}^P$ which
 preserves the punctures and their ordering such that $\sigma \circ \phi_1
 = \phi_2$ on $\partial \mathbb{D}$ (and hence on \disk) and $f_2^{-1}
 \circ \sigma \circ f_1$ is homotopic to the identity.  We then
 define
 \[  \widetilde{T}^P(\riem^P) = \{ (\riem^P,f_1,\riem_{1}^P,\phi)\} / \sim.  \]
\end{definition}

The rigged Teichm\"uller space $\widetilde{T}^P(\riem^P)$ possesses a
complex structure obtained in the following way.   First, we have a bijection
\[  \widetilde{T}^P(\riem^P) \cong T(\riem^B) / \DB \]
where $\DB$ and its action are defined in Section
\ref{TeichandMCG}. By Proposition \ref{DBaction},
$\widetilde{T}^P(\riem^P)$ inherits a complex structure from
$T(\riem^B)$ and the bijection is in fact a biholomorphism
\cite[Theorem 5.7 parts (3) and (4)]{RS05}.

Next define the map $\PDB :T(\riem^B) \rightarrow
\widetilde{T}^P(\riem^P)$ from \cite[Section 5.4]{RS05} in the following way (note that there, the map is denoted $P_{DB}$).  Fix a base
welding $\tau: \partial \mathbb{D}^n \rightarrow \partial \riem^B$
and a base Riemann surface $\riem^P= \riem^B \#_{\tau} \cdisk^n$.
Given $[\riem^B, h, \riem_1^B] \in T(\riem^B)$, let $\riem_1^P =
\riem_1^B \#_{h \circ \tau} \cdisk^n$, and define  $\tilde{h} : \riem^P \to \riem_1^P$ by 
\begin{equation} \label{eq:ftilde}
 \tilde{h} = \begin{cases}  h & \text{on } \riem^B \\
 \text{id} & \text{on } \mathbb{D}_0^n,
 \end{cases}
\end{equation}
Define $\tauP : \cdisk^n \to \riem^P$ by (\ref{eq:tauextension_definition}),
and set
\begin{equation}
\label{eq:PDBdef}
\PDB([\riem^B,h,\riem^B_1])= [\riem^P,\tilde{h},\riem^P_1,
\tilde{h} \circ \tauP] .
\end{equation}

This map satisfies $\PDB(p)=\PDB(q)$
if and only if $p$ is equivalent to $q$ under the action of $\DB$. Furthermore $\PDB$ is a holomorphic map which is locally a biholomorphism in the
sense that for any point in $w \in T(\riem^B)$ there's an open
set $U$ containing $w$ on which $\PDB$ is a biholomorphism onto an
open subset of $\widetilde{T}^P(\riem^P)$.
\begin{remark} \label{re:TtildePTBbiholomorphism}
  The biholomorphism between $T^B(\riem^B)/DB$ and $\widetilde{T}^P(\riem^P)$
  is given explicitly by 
\begin{align*}
 G:T^B(\riem^B)/DB & \rightarrow \widetilde{T}^P(\riem^P) \\
 [\riem^B,h,\riem^B_1] & \mapsto \PDB([\riem^B,h,\riem^B_1])
\end{align*}
That $G$ is well-defined and a
biholomorphism was established in \cite{RS05}.  It is not possible
to write the inverse explicitly, although we can say the
following:
\[  G^{-1}([\riem^P,f,\riem_*^P,\phi]) =
   [\riem^B,f_{\phi},\riem^P_* \backslash \phi(\mathbb{D}_0^n)] \]
where it is understood that the right hand side is a
representative of a point in $T^B(\riem^B)/DB$, and $f_{\phi}$ is a
quasiconformal map satisfying (1) $f_{\phi}$ is homotopic to $f$, and (2) $h \circ \tau=\phi$ on $\partial \mathbb{D} $.
Details are found in \cite[Theorems 5.5 and 5.6]{RS05}.
\end{remark}

\begin{remark} \label{re:PDB_no_inverse}
 Similarly, although $\PDB$ has a local holomorphic inverse, one cannot be completely explicit about it.
 Many of the difficulties in this paper can be traced to this fact.  However, we are able to write the restriction of the inverse to specified holomorphic curves.

  The local invertibility of $\PDB$ was established with the help of an existence theorem for quasiconformal maps, which is ultimately based on the $\lambda$-lemma.  See \cite[Sections 4 and 5.4]{RS05}.
\end{remark}

Next we define two fiber projections. Firstly, let
\begin{align} \label{eq:Fdefinition}
 \mathcal{F}: \widetilde{T}^P(\riem^P) & \rightarrow  T(\riem^P) \\
 [\riem^P,f,\riem_1^P,\phi] & \mapsto  [\riem^P,f,\riem_{1}^P].
 \nonumber
\end{align}
Note that in \cite{RS05} we call this map $\mathcal{F}_T$.
Secondly, we have the sewing map
\begin{align}
\label{sewingmap}
\C: T(\riem^B) & \longrightarrow T(\riem^P) \\
[\riem^B,h,\riem_{1}^B] & \longmapsto [\riem^B \#_{\tau}
\cdisk^n , \tilde{h}, \riem_{1}^B \#_{h \circ \tau} \cdisk^n]. \nonumber
\end{align}
It was proved in
\cite[Section 6]{RS05} that $\mathcal{F} \circ \PDB = \mathcal{C}$
and $\mathcal{C}$ is a holomorphic map.
Since $\PDB$ is a local biholomorphism it follows immediately that
$\mathcal{F}$ is a holomorphic map.

We define the fibers by
\[  F_B([\riem^P,f,\riem_{1}^P]) =
   \mathcal{C}^{-1}([\riem^P,f,\riem_{1}^P]).  \]
and
\[  F_P([\riem^P,f,\riem_{1}^P]) =
\mathcal{F}^{-1}([\riem^P,f,\riem_{1}^P]).  \]
 We will often denote
a particular fiber by $F_B$ or $F_P$, if there is no fear of
confusion.  The action of $\DB$ preserves each fiber $F_B$ since
$\mathcal{F} \circ \PDB = \mathcal{C}$ and thus we can conclude that
$F_P = G(F_B/ \DB)$.
\begin{remark}  $\mathcal{C}$ and $\mathcal{F}$, and hence the fibrations themselves,
 depend on the choice of $\tau$.
\end{remark}

For the convenience of the reader we summarize this section with a theorem.
\begin{theorem}
\label{th:sewing_summary}
\hfill
\begin{enumerate}
 \item $\widetilde{T}^P(\riem^P)$ possesses a complex structure,
  \item $\PDB$ is holomorphic, and for every point $w \in T(\riem^B)$, there is an open neighborhood $U$ of $w$ such that $\PDB$ is a biholomorphism onto its open image.
 \item $\mathcal{F} \circ \PDB=\mathcal{C}$.
 \item $\mathcal{F}$ and $\mathcal{C}$ are holomorphic and onto.
\end{enumerate}
\end{theorem}

\end{subsection}
\begin{subsection}{Fibers are complex submanifolds of $T(\riem^B)$}
 In this section we demonstrate that the fibers
are complex submanifolds of $T(\riem^B)$.  This requires the
construction of local sections.  The construction relies on a method
of Gardiner \cite{Gardiner} (see also Nag \cite{Nag}), who used Schiffer variation to
construct local coordinates on Teichm\"uller spaces of punctured Riemann
surfaces of finite type.

We outline some necessary facts about complex submanifolds of Banach
spaces.  These can be found in for example \cite[Section
1.6.2]{Nag}, and also \cite{Lang} in the differentiable setting.

Let $E_1$ and $E_2$ be Banach spaces, and $Y$ a complex Banach
manifold.  Let $U_1$ and $U_2$ be open subsets of $E_1$ and $E_2$
respectively.  We say that $g:U_1 \times U_2 \rightarrow Y$ is a
projection, if there is a map $h: U_1 \rightarrow Y$ such that $g =
h \circ \mbox{pr}_1$ where $\mbox{pr}_1: U_1 \times U_2 \rightarrow
U_1$ is the projection onto the first component.
\begin{definition}  Let $X$ and $Y$ be complex Banach manifolds.
$f:X \rightarrow Y$ is a holomorphic submersion if $f$ is
holomorphic, for all $x \in X$ there is a chart $(U,\phi)$ on $X$
with $x \in U$ and open sets $U_1 \subset E_1$ and $U_2 \subset E_2$
of Banach spaces $E_1$ and $E_2$, such that (1) $\phi:U \rightarrow
U_1 \times U_2$ is a biholomorphism and (2) there is a chart
$(V,\psi)$ with $f(x) \in V$, such that $\psi \circ f \circ
\phi^{-1}$ is a projection.
\end{definition}
A holomorphic fiber space is defined as follows.
\begin{definition}
A holomorphic fiber space is a pair of complex Banach manifolds
$(X,Y)$ together with a holomorphic, submersive and surjective map
$\pi : X \to Y$.
\end{definition}
It follows immediately from the definition of a submersion that the
fibers of a submersion are complex submanifolds.
\begin{lemma}
\label{submersion_submanifold} A holomorphic submersion $F: X\to Y$
is an open mapping and the fibers $F^{-1}(y)$ are complex
submanifolds of $X$.
\end{lemma}
We will use the following characterization of submersions.
\begin{lemma}
\label{submersion_section}
A holomorphic mapping $F: X \to Y$ between Banach spaces is submersive if and only if it possesses local holomorphic sections passing through every point $x \in X$.
\end{lemma}

We now briefly describe Schiffer variation of Riemann surfaces.  For
details see \cite{Gardiner} or \cite[Section 4.3]{Nag}.   Let $\riem$ be either a punctured Riemann
surface of type $(g,n)$
or a bordered Riemann surface of type $(g,n)$.  
Let $(V,\zeta)$ be a holomorphic chart on $\riem$ such that $\overline{\disk} \subset \zeta(V)$, and let $U = \zeta^{-1}(\disk)$ be a parametric disk on $\riem$.

For $\epsilon$ in a sufficiently small disk centered at $0 \in \mathbb{C}$, the map $v^{\epsilon}:\partial \disk \rightarrow
\mathbb{C}$ given by $v^{\epsilon}(z)=z+\epsilon/z$ is a biholomorphism on some neighborhood of $\partial \disk$.  Let $D^\epsilon$ denote the region bounded by
$v^{\epsilon}({\partial \mathbb{D}})$.  We obtain a new Riemann surface
$\riem^\epsilon$ as follows:
\[  \riem^\epsilon = (\riem \backslash U) \sqcup D^\epsilon/\sim  \]
where $x \in \partial U$ and $x' \in \partial D^\epsilon$ are
equivalent, $x \sim x'$, if $x'=v^{\epsilon} \circ \zeta(x)$. The complex structure on  $\riem^\epsilon$
is compatible with $D^\epsilon$ and $\riem \backslash \overline{U}$.

The quasiconformal map $w^\epsilon:\overline{\mathbb{D}} \rightarrow
D^\epsilon$ given by $w^\epsilon(z) = z+ \epsilon \bar{z}$ has
complex dilatation $\epsilon$.  Let $\nu^\epsilon:\riem \rightarrow
\riem^\epsilon$ be defined by
 \[
 \nu^\epsilon(x) =
 \begin{cases}
     x  & \text{if }  x \in \riem \backslash U \\
      w^\epsilon \circ \zeta(x) & \text{if }  x \in U .
 \end{cases}
 \]
This map is quasiconformal with dilatation $0$ on $\riem \backslash
\overline{U}$ and dilatation $\epsilon d\bar{\zeta}/d\zeta$ on $U$.

Let $\Omega$ be some polydisc centered at $0 \in \mathbb{C}^n$ and $\epsilon =(\epsilon_1,\ldots,
\epsilon_n) \in \Omega$.  The above variation procedure can be applied to $n$ non-overlapping parametric disks $U_1,\ldots, U_n$ to obtain a Riemann
surface $\riem^\epsilon$ and a quasiconformal map $\nu^\epsilon:\riem
\rightarrow \riem^\epsilon$.  We thus have a map
\begin{eqnarray*}
 S: \Omega & \rightarrow & T(\riem) \\
 \epsilon & \mapsto & [\riem,\nu^\epsilon,\riem^\epsilon] .
\end{eqnarray*}

The key result of \cite{Gardiner} and \cite[Theorem 4.3.2]{Nag} is the following:

\begin{theorem}
\label{SchifferTheorem} \hfill
\begin{enumerate}
\item The map $S$ is holomorphic for any type of Riemann surface $\riem$.
\item Let $\riem^P$ be a punctured surface of type $(g,n)$, and let $d = 3g-3+n$. For an essentially arbitrary choice of parametric disks $U_1, \ldots, U_n$, the parameters $(\epsilon_1,\ldots, \epsilon_d)$ provide local holomorphic coordinates for $T(\riem^P)$ in a neighborhood of $[\riem,\id, \riem]$. That is,  $S$ is a biholomorphism onto its image.
\end{enumerate}
\end{theorem}
The first result follows directly from the fact that dilatation of $\nu^{\epsilon}$ is holomorphic in $\epsilon$.
For the second result, it is non-trivial to prove that the variations give independent directions in $T(\riem^P)$.

Getting coordinates at an arbitrary point $[\riem^P,f,\riem_1^P]$ in \teich space follows by applying a change of base surface biholomorphism from $T(\riem^P)$ to $T(\riem^P_1)$ in the following way.
First we apply Schiffer variation to $\riem^P_1$. This gives a neighborhood of the base point in $T(\riem^P_1)$. Let $f_* : T(\riem^P_1) \to T(\riem^P)$ be the change of base surface biholomorphism corresponding to $f$. That is, $f_*([\riem^P_1, \nu^{\epsilon}, \riem^P_1]) = [\riem^P, \nu^{\epsilon}\circ f, \riem^P_1]$. Thus, the image under $f_*$ of the Schiffer neighborhood in $T(\riem^P_1)$ is the Schiffer neighborhood of $[\riem^P, f, \riem^P_1] \in T(\riem^P)$.

We can now show that
\begin{theorem} \label{th:Schiffer_section}
The map $\C$ (\ref{sewingmap}) possess local holomorphic sections through every point.
\end{theorem}
\begin{proof}
Let $[\riem^B, f, \riem_{1}^B]$ be an arbitrary point in $T(\riem^B)$ and let
$[\riem^P, \tilde{f}, \riem_{1}^P] = \C[\riem^B, f, \riem_{1}^B]$ where $\C$ is the sewing map defined in (\ref{sewingmap}).
Recall that $\riem_{1}^P = \riem^B_1 \#_{f \circ \tau} \cdisk^n$.
Let $d = 3g + 3 - n$. By Theorem \ref{SchifferTheorem}, we can choose $d$ disjoint disks on $\riem_{1}^B$ such that
performing Schiffer variation on $\riem_{1}^P$ using these disks results in a
biholomorphic map
\begin{align*}
S : \Omega & \longrightarrow S(\Omega) \subset T(\riem^P)  \\
\epsilon & \longmapsto [\riem^P, (\tilde{f})^{\epsilon}, \riem_{1}^{P, \epsilon}],
\end{align*}
where $(\tilde{f})^{\epsilon} =  \nu^{\epsilon} \circ \tilde{f}$, $\epsilon = (\epsilon_1, \ldots, \epsilon_d)$, and
$\Omega$ is an open neighborhood of $0 \in \mathbb{C}^d$.

Performing the Schiffer variation on $\riem_{1}^B$ using the same disks produces a holomorphic map
\begin{align*}
S^B : \Omega & \longrightarrow T(\riem^B)  \\
\epsilon & \longmapsto [\riem^B, f^{\epsilon}, \riem_{1}^{B, \epsilon}] ,
\end{align*}
where $f^{\epsilon} = \nu^{\epsilon} \circ f$.

By Theorem \ref{SchifferTheorem}, $\eta = S^B \circ S^{-1} :
S(\Omega) \to T(\riem^B)$ is holomorphic.  To show that it is a
section of $\C$ through $[\riem^B, f, \riem_{1}^B]$, it remains to
show that $\C \circ \eta$ is the identity.

Note that $f = f^{\epsilon}$ on $\partial \riem$, and by definition
of $\tilde{f}$, $(\tilde{f})^{\epsilon} = \widetilde{f^{\epsilon}}$.
Because the disks on which to perform Schiffer variation were chosen
to be away from the caps of $\riem^P_1$,
$$
\riem_{1}^{B,\epsilon} \#_{f^{\epsilon} \circ \tau} \cdisk^n = \riem_{1}^{P, \epsilon} .
$$
So we have
$$
(\mathcal{C} \circ \eta) \left([\riem^P, (\tilde{f})^{\epsilon}, \riem_{1}^{P, \epsilon}] \right) = \C \left([\riem^B, {f}^{\epsilon}, \riem_{1}^{B, \epsilon}]\right) =  [\riem^P, \widetilde{f^{\epsilon}}, \riem^{B,\epsilon}_1 \#_{f^{\epsilon} \circ \tau} \cdisk^n] = [\riem^P, (\tilde{f})^{\epsilon}, \riem^{P,\epsilon}_1]
$$
and thus $\mathcal{C} \circ \eta$ is the identity.
\end{proof}

We now have the following key results.

\begin{corollary}
\label{co:TBfiberspace}
The \teich space $T(\riem^B)$ is a holomorphic fiber space over $T(\riem^P)$ with fiber structure given by the sewing map
$$
\mathcal{C} : T(\riem^B) \longmapsto T(\riem^P) .
$$
\end{corollary}
\begin{proof}
Theorem \ref{th:sewing_summary} states that $\C$ is holomorphic and onto. Lemma \ref{submersion_section} and Theorem \ref{th:Schiffer_section} show that $\C$ is a submersion.
\end{proof}

From Lemma \ref{submersion_submanifold} we obtain:

\begin{corollary} \label{co:FBcomplexsubmanifold}
 The fibers $F_B$ are complex submanifolds of $T(\riem^B)$.
\end{corollary}

\begin{corollary} The fibers $F_P$ are complex submanifolds of $\widetilde{T}^P(\riem^P)$.
\end{corollary}
\begin{proof} From Theorem \ref{th:sewing_summary} $\mathcal{F}$ is holomorphic.  It possesses local holomorphic sections, since if $\rho:T(\riem^P) \rightarrow T(\riem^B)$ is a local holomorphic section
of $\mathcal{C}$, then $\PDB \circ \rho$ is a local holomorphic section of $\mathcal{F}$.
\end{proof}

\end{subsection}
\end{section}
\begin{section}{Local model of fibers}

\begin{subsection}{Complex structure on $\Oqc(\riem^P_1)$}
\label{se:CstructureOqc}
  The authors defined a complex structure on $\Oqc(\riem^P_1)$ in
  \cite{RSOqc},
  with the help of a model of the universal Teichm\"uller curve due
  to Teo \cite{Teo}.
  We outline the definition of the complex structure in this section.

  The complex structure $\Oqc(\riem^P_1)$ is locally biholomorphic to $\Oqc^n = \Oqc \times \cdots \times \Oqc$.  So we must first describe the complex structure on $\Oqc$.
 To do this, we define an injection of $\mathcal{O}_{qc}$ onto an
 open
 subset of a Banach space.  Consider the Banach space
 \[  A^1_\infty(\mathbb{D}) = \{ v(z): \mathbb{D} \rightarrow
 \mathbb{C} \;|\; v \text{ holomorphic}, \ ||v||_{1,\infty}=
 \sup_{z \in \mathbb{D}}
 (1- |z|^2) | v(z)| < \infty \}.  \]
 For $f \in \Oqc$ define
 \[  \mathcal{A}(f) = \frac{f''(z)}{f'(z)}.  \]
 We then have a natural one-to-one map
\begin{align}
\label{chi}
 \chi : \mathcal{O}_{qc} & \rightarrow  A^1_\infty(\mathbb{D})
 \oplus \mathbb{C}  \\
 f & \mapsto  \left( \mathcal{A}(f), f'(0) \right), \nonumber
\end{align}
where $A^1_\infty(\mathbb{D})\oplus \mathbb{C}$  is a Banach space with the direct sum norm
\[   ||(\phi,c)||=||\phi||_{1,\infty} +
|c|.  \] It was shown in \cite{RSOqc} that the image of $\Oqc$ is
open and thus $\Oqc$ inherits a complex structure from
$A^1_\infty(\mathbb{D})\oplus \mathbb{C}$.

 We now describe the complex structure on $\Oqc(\riem^P_1)$  in terms of
 local charts into $\Oqc^n$.
 Fix a point $(\phi_1,\ldots,\phi_n) \in \Oqc(\riem^P_1)$.
 For $i = 1,\dots,n$, let $D_i=\phi_i(\mathbb{D})$ be open sets in $\riem^P_1$.  Choose domains $B_i \subset \riem^P_1$ with the following properties: 1)
 $\phi_i(\overline{\mathbb{D}}) \subset B_i$, 2) $B_i \cap B_j = \emptyset$ for $i \neq j$ and 3) $B_i$ are open and simply connected.
 Let $\zeta_i:B_i \rightarrow \mathbb{C}$ be a local biholomorphic
 parameter such that $\zeta_i(p_i)=0$.  We then have that $\zeta_i
 \circ \phi_i \in \Oqc$.

 Let $K_i$ be a compact set, which is the closure of an open set
 containing $\zeta_i \circ
 \phi_i(\overline{\mathbb{D}})$,
 such that $\zeta_i^{-1}(K_i) \subset B_i$.
 By \cite[Corollary 3.5]{RSOqc}, for each $i$
 there is an open neighborhood $U_i$ of $\zeta_i \circ \phi_i$
 in $\Oqc$ such that $\psi_i(\overline{\mathbb{D}}) \subset K_i$ for
 all $\psi_i \in U_i$, and so $(\zeta_1^{-1} \circ \psi_1, \ldots,
 \zeta_n^{-1} \circ \psi_n)$ is an element of $\Oqc(\riem^P_1)$.

Thus, $U_1 \times \cdots \times U_n$ is an open subset of $\Oqc
\times \cdots \times \Oqc$ with the product topology.  Let
\[ V_i=\{ \,\zeta_i^{-1} \circ \psi_i \,|\, \psi_i \in U_i \,\};  \]
then sets of the form $V=V_1\times \cdots \times V_n$ form a base
of the topology of $\Oqc(\riem_1^P)$.  Let
\begin{align*}
 T_i: V_i  & \rightarrow \Oqc \\
 g & \mapsto \zeta_i \circ g.
\end{align*}
and note that $T_i(V_i)=U_i$. The local coordinates
\[  T=(T_1,\ldots,T_n):V \rightarrow U   \]
define a complex Banach manifold structure on $\Oqc(\riem^P_1)$.

\begin{remark}
\label{re:Tdomain}
In fact, fixing $B_i$ and $\zeta_i$, $i=1,\ldots,n$, $T$ is a
 valid chart on the set of {\it all} elements of $\Oqc(\riem^P_1)$ which map all $n$ copies of the disk into the corresponding set
 $B_i$.

 Let
 $U_i=\{ \psi_i \in \Oqc \st \psi_i(\overline{\mathbb{D}}) \subset \zeta_i(B_i) \}$ and observe that $U=U_1 \times \cdots \times U_n$ is open by \cite[Corollary 3.5]{RSOqc}.  Furthermore, the set $V=\{ \phi \in \Oqc(\riem_1^P) \st  \phi_i(\overline{\mathbb{D}}) \subset B_i \}$ is open since every point is contained in an open set of the form described above.  The map $T$ is clearly a biholomorphism on all of $V$. Hence, $(T,V)$ is a local coordinate.
\end{remark}

\end{subsection}

\begin{subsection}{The biholomorphism $L:\Oqc(\riem_*^P) \rightarrow F_P([\riem^P,f,\riem_*^P])$}

Fix a point $[\riem^P, f, \riem^P_*] \in T(\riem^P)$ and recall from Section \ref{se:fibration} that $F_P([\riem^P,f,\riem_*^P]) \subset \widetilde{T}^P(\riem^P)$ is the fiber over this point.
Define the map
\begin{align*}
 L : \Oqc(\riem_*^P) & \rightarrow F_P([\riem^P,f,\riem_*^P]) \\
 \phi & \mapsto  [\riem^P,f,\riem_*^P,\phi] .
\end{align*}

Our goal is to prove
\begin{theorem}
\label{th:Lbiholomorphism}
The map $L : \Oqc(\riem^P_*) \to F_P([\riem^P,f,\riem^P_*])$ is a biholomorphism.
\end{theorem}
\begin{proof}
We will prove this in three stages: first, we show that $L$ is a
bijection onto $F_P$ (Lemma \ref{le:Lbijection} ahead), second that $L$ is holomorphic (Lemma \ref{le:Lholo} ahead).  Lastly, we show that the inverse is holomorphic.

To show that the inverse is holomorphic,  in Section \ref{se:Lambda} we will construct near any point in $T(\riem^B)$ a local inverse $\Lambda$ to the lift $\PDB^{-1} \circ L$ where $\PDB^{-1}$ is a local inverse of $\PDB$.  By Theorem \ref{th:sewing_summary} it is enough to show that $\Lambda$ is holomorphic, which we will establish in Theorem \ref{th:Lambdaholo} ahead.
\end{proof}

Theorem \ref{th:Lbiholomorphism} has the following consequences.
\begin{corollary}
\label{co:OqcFBbiholo}
Fix any $[\riem^P,f,\riem_*^P] \in T(\riem^P)$.
The map
\begin{align*}
\Oqc(\riem_*^P) & \to F_B([\riem^P,f,\riem_*^P])/ \DB \\
\phi & \mapsto [\riem^B, f_{\phi}, \riem_*^P \setminus \phi(\disk_0^n)]
\end{align*}
is a biholomorphism,
where $f_{\phi}$ is the restriction of a map $f'_{\phi}: \riem^P \to \riem^P_1$ which is homotopic to $f$ and satisfies $f'_{\phi} \circ \tauP = \phi$.
\end{corollary}
\begin{proof}
  This follows from Theorem \ref{th:Lbiholomorphism} and Remark
  \ref{re:TtildePTBbiholomorphism}.
\end{proof}

\begin{corollary}
 The fibers $F_B([\riem^P,f,\riem_{1}^P])$ and $F_P([\riem^P,f,\riem_{1}^P])$ are
 locally biholomorphic to $\Oqc^n$ for any $[\riem^P,f,\riem_{1}^P]$ via the maps $ T \circ L^{-1} \circ \PDB$ and $T \circ L^{-1}$ respectively.
\end{corollary}

In the next few sections we require some facts regarding infinite-dimensional
holomorphy (see for example  \cite{Chae}, \cite[V.5.1]{Lehto} or
\cite[Section 1.6]{Nag}). Let $E$ and $F$ be Banach spaces and let
$U$ be an open subset of $E$.
\begin{definition}  A map $f \colon U \rightarrow F$ is holomorphic if
 for each $x_0 \in U$ there is a continuous complex linear map
 $Df(x_0) \colon E \rightarrow F$ such that
 \[  \lim_{h \rightarrow 0} \frac{|| f(x_0+h)-f(x_0) -
 Df(x_0)(h)||_F}{||h||_E} =0.  \]
\end{definition}
\begin{definition}
A map $f  \colon U \to F$ is called G\^ateaux holomorphic if $f$ is
holomorphic on complex lines. That is, if for all $a \in U$ and all
$x \in E$, the map $z \mapsto  f(a+zx)$ is holomorphic on $\{z \in
\mathbb{C} \st a + zx \in U \}$.
\end{definition}

\begin{theorem}[{\cite[p 198]{Chae}}]
\label{th:holo_equivalents} Let $f  \colon U \to F$. The following are
equivalent.
\begin{enumerate}
\item $f$ is holomorphic.
\item $f$ is G\^{a}teaux-holomorphic and continuous.
\item $f$ is G\^{a}teaux-holomorphic and locally bounded on $U$.
\end{enumerate}
\end{theorem}
A subset $H$ of the (continuous) dual space $F'$ is called \textit{separating} if for all non-zero $x \in F$ there exists $\alpha \in H$ such that $\alpha(x) \neq 0$.
The following theorem gives another characterization of
holomorphicity. See \cite{Grosse} for a statement in a more general
setting.
\begin{theorem}
\label{th:weak_holo}
Let $f : \Omega \to F$ be a
function on a domain $\Omega$ in $\Bbb{C}$.
If
\begin{enumerate}
\item $\alpha \circ f $ is holomorphic for each continuous linear functional
$\alpha$ from a separating subset of the dual space $F'$, and
\item $f$ is locally bounded,
\end{enumerate}
then $f$ is holomorphic.
\end{theorem}

We now turn to first step in the proof of Theorem \ref{th:Lbiholomorphism}.
\begin{lemma}
\label{le:Lbijection}
 $L$
is a bijection between $\Oqc(\riem_*^P)$ and $F_P([\riem^P, f,
\riem_*^P])$.
\end{lemma}
\begin{proof}
By the definitions of $\widetilde{T}^P(\riem^P)$ and $T(\riem^P)$,
every element in $F_P$ has a representative of the form $[\riem^P,
\tilde{f}, \riem_*^P, \phi]$. $L$ is thus clearly a surjection.

Assume that $L(\phi) = L(\psi)$. If $\mathcal{F}([\riem^P, f,
\riem_*^P, \phi]) = \mathcal{F}([\riem^P, f, \riem_*^P, \psi])$, then
there exists a biholomorphism $\sigma : \riem_*^P \to \riem_*^P$
that is homotopic to the identity.  Since $2g-2+n >0$ there are no
non-trivial automorphisms which are homotopic to the identity. Thus
$\sigma$ is the identity and so $\phi = \psi$.
\end{proof}

\begin{remark} The bijection is not canonical since
$\Oqc(\riem_*^P)$
 depends on the choice of representative
in the Teichm\"uller equivalence class. However, if $[\riem^P, f_1,
\riem_{1}^P] = [\riem^P, f_2, \riem_{2}^P]$ and $\sigma :
\riem_{1}^P \to \riem_{2}^P$ is the biholomorphism realizing the
equivalence, then we have the bijection $\sigma^* :
\Oqc(\riem_{1}^P) \to \Oqc(\riem_{2}^P)$ defined by
$\sigma^*(\phi_1) = \sigma \circ \phi_1$, and $L_1 = L_2 \circ
\sigma^*$. It will follow from Theorem \ref{th:Lbiholomorphism} (once
the proof is complete) that $\Oqc(\riem_1^P)$ and $\Oqc(\riem_2^P)$
are biholomorphic under $\sigma^*$.
\end{remark}
\end{subsection}
\begin{subsection}{$L$ is holomorphic}
\begin{lemma}
\label{le:Lholo}
 $L$ is holomorphic.
\end{lemma}
\begin{proof}

Fix $\phi_0 \in \Oqc(\riem^P_*)$. We will show that $L$ is
holomorphic in a neighborhood of $\phi_0$ by first proving that $L$
is G\^ateaux holomorphic and then that the lift of $L$ to
$L^{\infty}_{-1,1}(\riem^B)_1$ is locally bounded.

To simplify notation we assume that $\riem^P$ has only one puncture,
$p_1$.  The proof of the general case is identical with the
exception of the notation. Let $(T,V)$ be a chart on
$\Oqc(\riem_*^P)$ containing $\phi_0$ and let $B= B_1$, $D=D_1$ and
$\zeta = \zeta_1$ be as in Section \ref{se:CstructureOqc}. Recall
 that the map $\chi$ (\ref{chi}) defines the complex structure on $\Oqc$.

To show $L$ is G\^ateaux holomorphic we must show that
$$
L \circ T^{-1} \circ \chi^{-1} : A^1_{\infty}(\disk) \oplus \mathbb{C} \to
\widetilde{T}^P(\riem^P)
$$
is G\^ateaux holomorphic. Since the holomorphic structure on $\widetilde{T}^P(\riem^P)$ is obtained from $T(\riem^B)$ we proceed by producing a lift of $L \circ T^{-1} \circ \chi^{-1}$ to $T(\riem^B)$.

Let $\psi_0 = T(\phi_0) = \zeta \circ \phi_0 \in \Oqc$, $u_0 =
\mathcal{A}(\psi_0) \in A_1^{\infty}$, and choose an element $(v,c)
\in A_1^{\infty}(\disk) \oplus \mathbb{C}$. Let $N$ be an open
neighborhood of $0 \in \mathbb{C}$, and consider the map $N \to
\chi(T(V)) \subset A_1^{\infty} \oplus \mathbb{C}$ given by $t
\mapsto (u_0 + tv, q_t)$ where $q_t = \psi_0'(0)  + tc$. That is, we
have an complex line through $\chi(\psi_0) = (u_0, \psi_0'(0))$.

Let
$$
\psi_t = \chi^{-1}(u_0 + tv, q_t) .
$$
From the definition (\ref{chi}) of $\chi$ we have a differential
equation for $\psi_t$, whose solution is
\begin{equation}
\label{psitequation}
\psi_t(z)=\frac{q_t}{\psi_0'(0)} \int_0^z \psi_0'(\xi) \exp{\left(
t\int
       _0^{\xi} g(w)dw \right)} d\xi.
\end{equation}
It is clear from this expression that $\psi_t(z)$ is holomorphic in $t$ for fixed $z \in \disk$.

Now, $\psi_t \circ \psi_0^{-1}$ is a holomorphic motion of $\psi_0(\mathbb{D})$.  By the extended lambda-lemma \cite{Slodkowski}
it extends to a holomorphic motion of $\mathbb{C}$, and in particular the
continuous extension of $\psi_t \circ \psi_0^{-1}$ to
$\psi_0(\overline{\mathbb{D}})$ is a holomorphic motion.
Thus $\psi_t \circ \psi_0^{-1}$
restricted to $\psi_0(\partial \mathbb{D})$ is a holomorphic motion.

Let $\phi_t = T^{-1}(\psi_t) = \zeta^{-1} \circ \psi_t$, and let $A_t
= \overline{B} \setminus \phi_t(\disk)$ be the
annular region on $\riem^P_*$ bounded by $\partial B$ and
$\phi_t(\partial\mathbb{D})$. Applying \cite[Lemma 7.1]{RS05} we
obtain a holomorphic motion $H_t : \zeta(A_0) \to \zeta(A_t)$ such
that $H_t|_{\zeta(\partial B)}$ is the identity and
$H_t|_{\psi_0(\partial\mathbb{D})} = \psi_t \circ \psi_0^{-1}$.

Let $\riem^B_{*,t} = \riem_*^P \setminus \phi_t(\disk)$.
By \cite[Proposition 7.1]{RS05} the map $F_t : \riem^B_{*, 0} \to
\riem^B_{*, t}$ defined by
\begin{equation} \label{eq:Ftdefinition}
F_t =
\begin{cases}
\text{id} & \text{on } \riem_{*}^B \setminus A_0 \\
\zeta^{-1} \circ H_t \circ \zeta & \text{on } A_0 .
\end{cases}
\end{equation}
is quasiconformal and holomorphic in $t$ for fixed $z$. On
$\phi_0(\partial\mathbb{D})$, $F_t = \zeta^{-1} \circ \psi_t \circ
\psi_0^{-1} \circ \zeta = \phi_t \circ \phi_0^{-1}$. From
\cite[Proposition 7.2]{RS05}, $t \mapsto \mu(F_t)$ is a holomorphic
function $N \to L^{\infty}_{-1,1}(\riem^B_{*,0})_1$. Therefore, by 
the holomorphicity of the fundamental projection, $t\mapsto
[\riem^B_{*,0}, F_t ,\riem^B_{*,t}]$ is holomorphic.

Let $[\riem^B, h_0 ,\riem^B_{*,0}] = \PDB^{-1}([\riem^P,f,\riem^P_*,
\phi_0])$ where $\PDB^{-1}$ is a local inverse of $\PDB$ and $\tilde{h}_0 \circ \tilde{\tau} = \phi_0 $ from the definition of $\PDB$ in (\ref{eq:PDBdef}). For $g: \riem^B \to \riem^B_{*,0}$, let
\begin{align*}
{g}_* : T(\riem^B_{*,0}) & \to T(\riem^B) \\
[\riem^B_{*,0}, h_1, \riem^B_1] & \mapsto [\riem^B, h_1 \circ g,
\riem^B_1].
\end{align*}
be the change of base point biholomorphism (see \cite[Sections 2.3.1
and 3.2.5]{Nag}). In particular, $g_* ([\riem^B_{*,0}, F_t,
\riem^B_{*,t}])   = [\riem^B, \, F_t \circ g, \riem^B_{*,t}]$ is a
biholomorphism, and therefore the map $t \mapsto [\riem^B, \, F_t
\circ h_0, \riem^B_{*,t}]$ from $N$ into $T(\riem^B)$ is
holomorphic. Moreover, $t \mapsto \mu(F_t \circ h_0)$ is also
holomorphic.

From the boundary values of $F_t$ and $h_0$ we see that $F_t \circ h_0
\circ \tau = \phi_t$ on $\partial \disk$. Furthermore, extending
$F_t$, $h_0$ and $\tau$ by the identity to the caps as in equation
(\ref{eq:ftilde}), we have that $\tilde{F}_t \circ \tilde{h}_0 \circ
\tilde{\tau}$ is homotopic to $f$ (after identifying $\riem^B_{*,t}
\# \cdisk$ with $\riem^P_*$).   Thus
$$
\PDB([\riem^B, \, F_t \circ h_0,\riem^B_{*,t}]) = [\riem^P,
\tilde{F}_t \circ \tilde{h}_0,\riem^B_{*,t} \#_{\phi_t} \cdisk,\phi_t]= [\riem^P, f, \riem^P_*, \phi_t]
$$
and so
\begin{align*}
N & \to \widetilde{T}(\riem^P) \\
t & \mapsto (L \circ T^{-1} \circ \chi^{-1})(h_0 + tg, q_t) = [\riem^P, f, \riem^P_*, \phi_t]
\end{align*}
is holomorphic. That is, $L \circ T^{-1} \circ \chi^{-1}$ is G\^ateaux holomorphic.

We have that 1) the fundamental projection
$\Phi:L^\infty_{-1,1}(\Sigma^B)_1 \rightarrow T(\Sigma^B)$ is
holomorphic and possesses local holomorphic sections and 2)
$\PDB:T(\Sigma^B) \rightarrow \widetilde{T}^P(\Sigma^P)$ is
holomorphic and possesses local holomorphic sections.  Thus there
is an open neighborhood $M$ of $[\riem^P,f,\riem^P_*,\phi_0] \in
\widetilde{T}^P(\riem^P)$ and a local holomorphic section
$\sigma:M \rightarrow L^\infty_{-1,1}(\Sigma^B)$ of $\PDB \circ
\Phi$. Continuity in $t$ guarantees that for $|t|$ sufficiently
small, $[\riem^P,f,\riem^P_*,\phi_t] \subset M$. The function
$\sigma \circ L$ is G\^ateaux holomorphic and  locally bounded
since $\sigma$ maps into the open unit ball.  Thus, by Theorem
\ref{th:holo_equivalents}, $\sigma \circ L$ is holomorphic, and
therefore $L = \mathcal{P} \circ \Phi \circ \sigma \circ L$ is
holomorphic.
\end{proof}

\begin{remark}
\label{pointevalholo}
For any fixed $z \in \disk$, the point evaluation map $\Oqc \to \mathbb{C}$ given by $f \mapsto f(z)$ is holomorphic. The proof is an immediate consequence of Theorem \ref{th:holo_equivalents} noting that G\^ateaux holomorphy follows from equation (\ref{psitequation}), and continuity is proved in \cite[Corollary 3.4.]{RSOqc}.
This result is also included in the proof of \cite[Lemma 3.10]{RSOqc} but it is not mentioned explicitly. 
\end{remark}

\end{subsection}
\begin{subsection}{The local inverse of $L$} \label{se:Lambda}
In order to show that $L^{-1}$ is biholomorphic, we show that  $\PDB^{-1} \circ L:\Oqc(\riem^P_*) \rightarrow T(\riem^B)$ has a local holomorphic inverse
for any local inverse $\PDB^{-1}$ of $\PDB$.  The description of the inverse to $\PDB^{-1} \circ L$ is somewhat lengthy and deserves its own section.  In fact we are only able to explicitly describe the inverse on specified holomorphic curves.  The source of the trouble can be partly traced to the fact that local inverses  of $\PDB$
cannot be explicitly defined (see Remark \ref{re:PDB_no_inverse}).

It is necessary to make the following change of base point. Recall that $\riem^P = \riem^B\#_\tau  \cdisk^n$ as in Section \ref{se:fibration}, which we now think of as a punctured surface.
Let $\mathbb{U}$ be the upper half-plane, and choose a Fuchsian group $G$ such that $\riem_G = \mathbb{U}/G$
is an $n$-punctured surface
biholomorphic to $\riem^P$. Let 
 $$\alpha:\riem^P \rightarrow
 \riem_G
 $$ 
 be a fixed biholomorphism.

 Let $A$ be an open set of
 $T(\riem^B)$ such that $\PDB |_A$ is a biholomorphism.
 Let
 \begin{equation}
 \label{eq:A_Bdef}
 A_B=F_B \cap A ,
 \end{equation}
 where $F_B$ is the fiber in $T(\riem^B)$
 above the fixed point $[\riem^P,f,\riem_*^P]$ (see Section \ref{se:fibration}).
 
Given $[\riem^B,h,\riem_1^B] \in A_B$ we have that
$\PDB([\riem^B,h,\riem_1^B]) = ([\riem^P,\tilde{h}, \riem_1^B \#_{h
\circ
   \tau} \cdisk^n])$.  The change of base point biholomorphism $(\alpha^*)^{-1} : T(\riem^P) \to T(\riem_G)$ is defined by
$$
(\alpha^*)^{-1} ([\riem^P,\tilde{h}, \riem_1^B \#_{h \circ
   \tau} \cdisk^n])=[\riem_G, \tilde{h} \circ \alpha^{-1},
  \riem_1^B \#_{h \circ \tau} \cdisk^n].
$$ 

Now define the Beltrami differential
$$ \mu= \mu(\tilde{h} \circ \alpha^{-1}) =  \begin{cases}
  \mu(h \circ \alpha^{-1}) &  \text{on }   \alpha(\riem^B) \subset \riem_G \\
  0 & \text{on }  \alpha(\mathbb{D}).
\end{cases}
$$
in $L^{\infty}_{-1,1}(\riem_G)_1$. 
Let $L^{\infty}(\mathbb{U},G)_1$ be the space of Beltrami differentials compatible with $G$, and
following \cite[p. 51]{Nag} we identify $\mu$ with its  unique lift to $L^\infty(\mathbb{U},G)_1$. Let $w^\mu$ be the
unique solution to the Beltrami equation on $\mathbb{C}$ fixing $0$,
$1$ and $\infty$ and having dilatation $0$ on the lower half-plane.  Let $G^\mu = w^\mu \circ G \circ (w^{\mu})^{-1}$ and
\[  \riem^\mu = w^\mu(\mathbb{U})/G^\mu.  \]

Let $T(G) = L^{\infty}(\mathbb{U},G)_1 / \sim $, where $\mu \sim
\nu$ if and only if $w^\mu = w^{\nu}$ on $\mathbb{R}$, be the
Teichm\"uller space of $G$. Let $\pi^{\mu}: w^{\mu}(\mathbb{U}) \to
w^{\mu}(\mathbb{U}) / G^{\mu}$ be the canonical projection. It is a standard fact (see \cite[Sections 2.2.2 and 3.3.1]{Nag}) that:
\begin{lemma}
\label{le:indep_mu}
If $[\mu] = [\nu]$ in $T(G)$ then
$w^{\mu}(\mathbb{U}) = w^{\nu}(\mathbb{U})$, $G^{\mu} = G^{\nu}$, and $\pi^{\mu} = \pi^{\nu}$.
\end{lemma}

\begin{proposition}
\label{pr:riem_mu_fixed}
 The equivalence class $[\mu]$ and the Riemann surface $\riem^{\mu}$ are independent of $[\riem^B,h,\riem_1^B] \in A_B$.
\end{proposition}
\begin{proof}
By definition of the fiber $F_B$,
$
[\riem^P, f, \riem^P_*] = [\riem^P, \tilde{h}, \riem^B_1 \#_{h \circ \tau} \cdisk]
$, and applying $(\alpha^*)^{-1}$ leads to
\begin{align*}
[\riem_G, f \circ \alpha^{-1},\riem^P_*] & = [\riem_G, \tilde{h} \circ \alpha^{-1}, \riem^B_1 \#_{h \circ \tau} \cdisk] \\
\end{align*}
Therefore  $ \mu = \mu(
 \tilde{h} \circ \alpha^{-1})$ is equivalent to  the fixed element
 $\mu(f \circ \alpha^{-1})$. The result now follows from Lemma \ref{le:indep_mu}.
\end{proof}

We may therefore let $\sigma :\riem^\mu \rightarrow \riem_*^P$ be a fixed
biholomorphism, and define
\begin{align}
\label{eq:Lambdadef}
 \Lambda: A_B & \rightarrow \Oqc(\riem^P_*) \\
 [\riem^B,h,\riem_1^B] & \mapsto \sigma \circ f^{\mu} \circ \alpha \circ \tauP \nonumber
\end{align}
where $\mu=\mu(\tilde{h}\circ \alpha^{-1})$ and $f^\mu$ is the
unique quotient map
\[  f^\mu: \riem_G \rightarrow \riem^\mu.  \]
corresponding to $w^\mu$.
\begin{remark} \label{re:homotopic}
 The argument above also shows that for any $[\riem^B,h_1,\riem^B_1]$ and $[\riem^B,h_2,\riem^B_2]$ in $A_B$, the corresponding $w^{\mu_1}$ and $w^{\mu_2}$
 are homotopic rel $\partial \mathbb{U}$, and similarly
 $f^{\mu_1}$ and $f^{\mu_2}$ are homotopic in $\riem_G$ \cite[Sections 2.2.2 and 3.3.1]{Nag}.
\end{remark}

\begin{remark}
\label{mu_basepoint_holo}
The map 
\begin{align*}
L^{\infty}_{-1,1}(\riem^B) & \to L^{\infty}_{-1,1}(\riem_G) \\
\mu(h) & \mapsto \mu(\tilde{h}\circ \alpha^{-1})
\end{align*}
is holomorphic. This follows from the fact that
$\mu(h) \mapsto \mu(\tilde{h})$ is holomorphic (see \cite[Lemma 6.2]{RS05}), and the map $\mu(\tilde{h}) \mapsto \mu(\tilde{h} \circ \alpha^{-1})$ is holomorphic by the change of base point holomorphicity (see \cite[Sections 2.3.1
and 3.2.5]{Nag}).
\end{remark}

\begin{proposition}
\label{pr:LambdaWellDef}  If $(\riem^B,h_1,\riem_1^B)$ is equivalent to $(\riem^B,h_2,\riem_2^B)$ in $A_B \subset T(\riem^B)$, then the corresponding maps $f^{\mu_1}$ and $f^{\mu_2}$ are equal on $\alpha(\partial \riem^B)$.  In particular,
$\Lambda$ is well-defined.
\end{proposition}
\begin{proof}
 Let $[\riem^B,h_1,\riem_1^B]=[\riem^B,h_2,\riem_2^B]$, so that there exists a biholomorphism $\gamma:\riem_1^B \rightarrow \riem_2^B$ such that $h_2^{-1} \circ \gamma \circ h_1$ is homotopic to the identity rel $\partial \riem^B$.  Since the dilatation of $\gamma \circ h_1$ and $h_1$ are equal,  composition by $\gamma$ does not change the resulting $f^\mu$ in the definition of $\Lambda$.  Thus we may absorb
 $\gamma$ into $h_1$ and assume that $\riem_1^B=\riem_2^B$ and $h_2^{-1} \circ h_1$ is homotopic to the identity rel $\partial \riem^B$.
In particular, $h_1 = h_2$ on $\partial \riem^B$ and moreover $\tilde{h}_1$ is homotopic to $\tilde{h}_2$ on $\riem^P$.

Now let $\sigma:\riem^\mu \rightarrow \riem_*^P$ be the fixed biholomorphism in the definition of $\Lambda$.  Let $f^{\mu_i}: \riem_G \rightarrow \riem^\mu$ be the pair of maps with dilatations $\mu_i = \mu(\tilde{h}_i \circ \alpha^{-1})$ respectively.

Since $\tilde{h}_i \circ \alpha^{-1}$ and $f^{\mu_i}$ have the same
dilatations for $i=1,2$, there is a pair of biholomorphic maps
$\delta_i:\riem_i^B \#_{h_i \circ \tau} \cdisk^n \rightarrow
\riem^\mu$ such that
\begin{equation} \label{eq:deltahalphaf}
 \delta_i \circ \tilde{h}_i \circ \alpha^{-1} = f^{\mu_i}
\end{equation}
on
$\riem_G$. Because $h_1 = h_2$ on $\partial \riem^B$ and $\riem^B_1
= \riem^B_2$, $\riem_1^B \#_{h_1 \circ \tau} \cdisk^n = \riem_2^B
\#_{h_2 \circ \tau} \cdisk^n$.

By Remark \ref{re:homotopic}, we know that $(f^{\mu_2})^{-1} \circ
f^{\mu_1}$ is homotopic to the identity on $\riem_G$. Thus
$\delta_2^{-1} \circ \delta_1 $ is homotopic to the identity.  Since
$\riem^P$ is of type $(g,n)$ with $2g-2+n > 0$ the biholomorphism
$\delta_2^{-1} \circ \delta_1 $ must be the identity. We can
conclude that $\delta_1=\delta_2$. Therefore, since $h_1 = h_2$ on
$\partial \riem^B$, by (\ref{eq:deltahalphaf}) $f^{\mu_1} =
f^{\mu_2}$ on  $\alpha(\partial \riem^B)$, so
$$\Lambda([\riem^B, h_1, \riem_1^B]) = \Lambda([\riem^B, h_2, \riem_2^B]) .
$$

\end{proof}

Let $\PDB^{-1}$ denote a locally defined inverse of $\PDB$ (see
(\ref{eq:PDBdef})) in a neighborhood of $[\riem^P, f, \riem^P_*,
\phi_0]$. We need to show that
$$
\Lambda \circ \PDB^{-1} \circ L = \id
$$
on some neighborhood of $\phi_0$.   For notational simplicity we
work with the case of a single puncture on $\riem^P$.
Let $[\riem^B, h_0, \riem^B_{*, 0}] = \PDB^{-1} ([\riem^P, f,
\riem^P_*, \phi_0])$ and recall that $ \riem^B_{*,0} = \riem^P_*
\setminus \phi_0(\disk)$ and $h_0 \circ \tau = \phi_0$ on
$\partial \disk$.

Choose a compact set $E \subset \riem^P_*$ such that $E \subset \phi_0(\disk)$. Let $V_0 \subset \Oqc(\riem^P_*)$ be an open neighborhood  of $\phi_0$ such that for all $\phi \in V_0$, $E \subset \phi(\disk)$. This is possible by \cite[Corollary 3.4]{RSOqc}.

Let $\iota_0 : \riem^B_{*, 0} \#_{h_0 \circ \tau} \cdisk \to
\riem^P_*$ be the biholomorphism defined by
\begin{equation} \label{eq:iotadefinition}
\iota_0(z) =
\begin{cases}
z & z \in \riem^B_{*, 0} \\
\phi_0(z) & z \in \disk .
\end{cases}
\end{equation}
Since $h_0 \circ \tau = \phi_0$ on $\partial \mathbb{D}$, we have
\begin{equation}  \label{eq:temp}
 \iota_0 \circ \tilde{h}_0 \circ \tilde{\tau} = \phi_0.
\end{equation}
Let
$$\sigma = \iota_0 \circ \delta_0^{-1} : \riem^{\mu} \to \riem^P_*,
$$
where $\delta_0 :\riem_{*, 0}^B \#_{h_0 \circ \tau} \cdisk
\rightarrow \riem^\mu$ is the unique biholomorphism  satisfying (\ref{eq:deltahalphaf}).  
Use the map $\sigma$ for the biholomorphism in the definition of
$\Lambda$. Now let $\psi = \Lambda([\riem^B, h_0, \riem^B_{*,0}])$.
It follows directly that $\psi = \phi_0$:
\begin{align*}
\psi & = \sigma \circ f^{\mu_0} \circ \alpha \circ \tauP \\
& = \iota_0 \circ \delta_0^{-1} \circ ( \delta_0 \circ \tilde{h}_0 \circ \alpha^{-1} ) \circ \alpha \circ \tauP \\
& = \iota_0 \circ \tilde{h}_0 \circ \tauP.
\end{align*}
Thus by equation (\ref{eq:temp}) $\psi=\phi_0$.

For $\phi \in V_0$, let $[\riem_B, h_{\phi}, \riem^B_{*, \phi}] =
(\PDB^{-1} \circ L)(\phi)$ where $\riem^B_{*, \phi} = \riem^P_* \setminus \phi(\disk_0)$. We eventually want to show that for any $\phi
\in V_0$, $\Lambda([\riem_B, h_{\phi}, \riem^B_{*, \phi}]) = \phi$.

Choose a holomorphic curve $\phi_t \in V_0$ joining $\phi_0$ and
$\phi$. That is,  $\phi_1 = \phi$ and $\phi_0$ is as above. By Remark \ref{pointevalholo}, $\phi_t(z)$ is holomorphic in $t$ and so the construction in the proof of Lemma \ref{le:Lholo} can be repeated. Let
$\riem^B_{*, t}=\riem^B_{*, \phi_t}$ and define $F_t \circ h_0 : \riem^B
\to \riem^B_{*, t}$ where $F_t$ is as in (\ref{eq:Ftdefinition}).
As before, $t \mapsto \mu(F_t \circ h_0)$ is holomorphic. Let $h_t = F_t
\circ h_0$, $\mu_t = \mu(\tilde{h}_t \circ \alpha^{-1})$, and let $\delta_t : \riem^B_{*, t} \#_{h_t \circ \tau}
\overline{\mathbb{D}} \rightarrow \riem^\mu$ be the biholomorphism defined by
$$
\delta_t = f^{\mu_t} \circ \alpha \circ \tilde{h}_t^{-1}
$$ 
as in (\ref{eq:deltahalphaf}).   
Define $\iota_t$ as in (\ref{eq:iotadefinition}) by replacing $\phi_0$
with $\phi_t$.

\begin{lemma}
\label{le:delta_t_holo}
The map
$\delta_t : \riem^B_{*,t} \#_{h_t \circ \tau} \cdisk \to \riem^{\mu}$ is holomorphic in $t$ for $z \in \disk$.
\end{lemma}
\begin{proof}
Since $t \mapsto \mu(h_t)$ is holomorphic in $t$, Remark \ref{mu_basepoint_holo} shows that
$\mu_t = \mu(\tilde{h}_t \circ \alpha^{-1})$ is holomorphic in $t$.
Thus for fixed $z$, $w^{\mu_t}(z)$ is holomorphic in $t$. Because
$\pi^{\mu}$ is independent of $\mu$ by Lemma \ref{le:indep_mu} and
Proposition \ref{pr:riem_mu_fixed},  $f^{\mu_t}(z)$ is also
holomorphic in $t$ for fixed $z$.

For $z \in \disk$,
$$\delta_t(z) = (f^{\mu_t} \circ \alpha \circ \tilde{h}_t^{-1})(z)  = (f^{\mu_t} \circ \alpha)(z) $$
from the definition of $\delta_t$, and the conclusion follows immediately.
\end{proof}
\begin{lemma}
\label{sigmadef}
$\sigma = \iota_0 \circ \delta_0^{-1} = \iota_t \circ \delta_t^{-1}$
\end{lemma}

\begin{proof}
Let
$$ \beta_t  = \sigma \circ \delta_t \circ \iota_t^{-1} : \riem^P_* \to \riem^P_*.$$
We will show that $\beta_t$ is the identity.  First we claim that
$\beta_t$ is holomorphic in $t$ for $z \in E$. We have that
$\phi_t(z)$ is holomorphic as a function of $t$ and $z$ and thus so
is $\phi^{-1}_t(z)$ by a direct application of the implicit function theorem. For $z \in E$, and using $\iota_t \circ \tilde{h}_t \circ \tilde{\tau} = \phi_t$ from (\ref{eq:temp}), 
\begin{align*}
\beta_t(z) & = (\sigma \circ f^{\mu_t} \circ \alpha \circ \tilde{h}_t^{-1} \circ \iota_t^{-1})(z) \\
& = (\sigma \circ f^{\mu_t} \circ \alpha \circ \tauP \circ \phi_t^{-1})(z)
\end{align*}
which is holomorphic in $t$.

 The fact that $\beta_t$ is the identity follows from the following four observations.
 (1) By Hurwitz's theorem, the automorphism group of $\riem^P_*$ is finite. (2) $\beta_t$ is the identity if and only if $\beta_t |_E$ is the identity. (3)  For $z \in E$, $\beta_t(z)$ is continuous in $t$. (4) From the definition of $\sigma$, $\beta_0$ is the identity.
\end{proof}

\begin{theorem}
\label{th:LambdaisLinv} For all $\phi \in V_0$, and any local
inverse $\PDB^{-1}$ of $\PDB$, $(\Lambda \circ \PDB^{-1} \circ
L)(\phi) = \phi$.
\end{theorem}
\begin{proof}
We join $\phi$ to $\phi_0$ by a holomorphic curve $\phi_t$ such that $\phi_1 = \phi$.
From Lemma \ref{sigmadef} and equation (\ref{eq:deltahalphaf}) we have:
\begin{align*}
(\Lambda \circ  \PDB^{-1} \circ L)(\phi_t) & = \Lambda([\riem^B, h_t, \riem^B_{*,0}]) \\
& = \sigma \circ f^{\mu_t} \circ \alpha
\circ \tauP \\
& = (\iota_t \circ \delta_t^{-1})  \circ (\delta_t \circ \tilde{h}_t \circ \alpha^{-1}) \circ \alpha \circ \tauP \\
& = \iota_t \circ \tilde{h}_t \circ \tauP \\
& = \phi_t
\end{align*}
where the last equality follows as both $\tilde{\tau}$ and $\tilde{h}_t$ are the identity on $\disk$.
Setting $t=1$ completes the proof.
\end{proof}
\end{subsection}
\begin{subsection}{$L^{-1}$ is holomorphic}

\begin{lemma} \label{le:nearness}
Let $[\riem^B, h_0, \riem^B_{*, 0}] = \PDB^{-1} ([\riem^P, f,
\riem^P_*, \phi_0])$ for some local inverse $\PDB^{-1}$ of $\PDB$.
Let $B$ be an open set in $\riem^P_*$ such that
 $(\sigma \circ f^{\mu_0} \circ \alpha \circ \tauP)(\cdisk) \subset B$.
 There exists a neighborhood $A_0 \subset F_B$ of $[\riem^B,h_0,\riem^B_{*,0}]$ such
 that
 $$
\Lambda \left( [\riem^B,h,\riem^B_1] \right)(\cdisk) \subset B
 $$
for all $[\riem^B,h,\riem^B_1] \in A_0$.
\end{lemma}

\begin{proof}
Let $\gamma$ be a local holomorphic section of the fundamental
projection $\Phi$ in a neighborhood of
$[\riem^B,h_0,\riem^B_{*,0}]$, and choose $U$ to be an open set in
the domain of $\gamma$ containing this point.
Given $u \in U$ we choose a representative $[\riem^B,h_u,\riem^B_1] = u$ such that $\mu(h_u) = \gamma(u)$.

By Remark \ref{mu_basepoint_holo}, $\mu(h_u) \mapsto \mu(\tilde{h}_u \circ \alpha^{-1})$ is holomorphic. Therefore, since $\gamma$ is holomorphic, the map
\begin{align*}
U & \to L^{\infty}_{-1,1}(\riem_G) \\
u & \mapsto  \mu(\tilde{h}_u \circ \alpha^{-1})
\end{align*}
is holomorphic and in particular continuous.

We now show in general that $w^\mu(z)$ and hence $f^\mu(z)$ are jointly continuous in $\mu$ and $z$. Let $\mathbb{D}_r=\{z:|z|<r\}$ for some $r>1$.  Fix $\mu_0$ and $z_0$.  By \cite[Theorem 4.7.4]{Hubbard}, for any $\epsilon>0$ there is a $\delta_1$ so that if $\|\mu-\mu_0\|<\delta_1$ then $|w_\mu(z)-w_{\mu_0}(z)|<\epsilon/2$ for all $z \in \overline{\mathbb{D}}_r$. On the other hand, since $w_{\mu_0}$ is continuous in $z$ there is a $\delta_2$ such that $|w_{\mu_0}(z)-w_{\mu_0}(z_0)|<\epsilon/2$ for all 
$|z-z_0|<\delta_2$.  We can assume that $\delta_2$ is small enough that the disk of radius $\delta_2$ centered on $z_0$ is contained in $\mathbb{D}_r$.
Thus for $|z-z_0|<\delta_2$ and $\|\mu-\mu_0\|<\delta_1$
\[  |w_{\mu}(z)-w_{\mu_0}(z_0)| \leq |w_\mu(z)-w_{\mu_0}(z)|+|w_{\mu_0}(z)-w_{\mu_0}(z_0)| <\epsilon.  \]
So $w^\mu(z)$ is jointly continuous in $\mu$ and $z$ for all $z \in \mathbb{D}$ and hence so is $f^\mu(z)$.  

By applying this fact with $\mu = \mu(\tilde{h} \circ \alpha^{-1})$, we have that
$$
(u,z) \mapsto f^{\mu}(z)
$$ 
is continuous. 

 In the following let $D(\zeta,R)$ denote the disc of radius $R$
 centered on $\zeta$.  Let $B$ be as in Section
 \ref{se:CstructureOqc} (recall that we are assuming that there is a single
 puncture).
Since each $(\sigma \circ f^\mu \circ \alpha)(\partial \riem^B)$ is
compact, there is an
 $r$ so that $D(\xi,r) \subset B$ for all $\xi \in (\sigma \circ f^\mu \circ \alpha)(\partial \riem^B)$.
 By the continuity of $f^\mu(\zeta)$ in both $\mu$ and $\zeta$ for each
 $\zeta \in \alpha(\partial \riem^B)$, and by letting $\xi=(\sigma \circ f^\mu)(z)$,
 one can choose an open neighborhood
 $U_\xi = D(z,\delta_\xi) \times W_\xi$ of $(z,[\riem^B,h_0,\riem^B_{*, 0}])$ so that
 $f^\mu(\zeta) \subset B(\xi,r)$ for all $\zeta \in D(z,\delta_\xi)$ and
 $[\riem^B,h_u,\riem_1^B] \in W_\xi$.   Since $\alpha(\partial
 \riem^B)$ is compact, its open cover by the union of
 $D(z,\delta_\xi)$ has a finite
 subcovering $D(z_1,\delta_{\xi_1}),\ldots, D(z_m,\delta_{\xi_m})$.
 The open set $A_0=W_{\xi_1} \cap \cdots
 \cap W_{\xi_m}$ has the desired properties.
\end{proof}

Fix a point $z_0 \in \alpha(\disk_0)$. Choose $Q \subset \disk_0$ to be an open neighborhood of $z_0$ such that there exists a local holomorphic section $s$ of $\pi_G : U \to U/G =
\riem_G$ defined on $(\alpha \circ \tau)(Q)$. Let $\pi = \pi^{\mu}$, which is
independent of $\mu$ by Lemma \ref{le:indep_mu} and Proposition
\ref{pr:riem_mu_fixed}. For $z \in Q$ we can now write $\Lambda$ in terms of
$w^\mu$ and fixed maps as follows:
\begin{equation}
\label{eq:Lambda_wmu}
\Lambda([\riem^B, h, \riem_1^B])
 = \sigma \circ (\pi \circ w^{\mu} \circ s) \circ \alpha \circ \tauP .
\end{equation}

With the aid of Theorem \ref{th:weak_holo}, we can now proceed with
the proof that $\Lambda$ is holomorphic.

\begin{lemma}
\label{lemma_phit_holo}
Let $t \mapsto [\riem^B, h_t, \riem_t^B]$ be a holomorphic curve in
$A_B$. For any $z \in \disk$, $\phi_t(z)=
\Lambda([\riem^B, h_t, \riem_t^B])(z)$ and all its derivatives in $z$
are holomorphic in $t$.
\end{lemma}
\begin{proof}
We first show that the claim holds for a neighborhood of any fixed $z_0$ in $\mathbb{D}_0$.  By the existence of holomorphic sections of the
fundamental projection the curve $t \mapsto [\riem^B, h_t,
\riem_t^B]$ is the image of a holomorphic curve in
$L^{\infty}_{-1,1}(\riem^B)_1$. We can thus assume without loss of
generality that our representatives $(\riem^B,h_t,\riem_t^B)$ are
such that $t \mapsto \mu(h_t)$ is holomorphic in $t$.
Thus the maps $t \mapsto  \mu_t =
\mu(\tilde{h}_t)$ and $t \mapsto w^{\mu_t}(z)$ are holomorphic in
$t$  as in the proof of Lemma \ref{le:delta_t_holo}.

Let $C = \alpha(\disk_0)$ be the cap on $\riem_G$.
Since $\mu_t$ is zero on $C$, $w^{\mu_t}(z)$ is holomorphic in $z$
for $z \in s(C)$. Therefore, it is a holomorphic function of $t$ and
$z$ and hence all its derivatives are also holomorphic functions of
both variables. The statement for $\phi_t$ then follows from
(\ref{eq:Lambda_wmu}) for all $z$ in some neighborhood $Q$ of $z_0$. This proves the claim for $z \neq 0$.

To prove the claim for $z_0 = 0$, observe that for all $t$,
$\phi_t(0)=p$. So in fact we know that $\phi_t(z)$ is holomorphic
in $z$ for fixed $t$ for all $z \in \mathbb{D}$.  Furthermore,
clearly $\phi_t(0)$ is holomorphic in $t$.  Thus $\phi$ is
holomorphic in $t$ and $z$ separately and so by Hartog's theorem
$\phi$ is jointly holomorphic in both $t$ and $z$.  Thus all the derivatives
of $\phi$ with respect to $z$ are holomorphic in $t$ for any fixed
$z \in \mathbb{D}$.
\end{proof}

Fix the
point $[\riem^B,h_0,\riem^B_{*,0}] \in F_B$, and
let $B_i$ $i=1,\ldots, n$
and $A_0$ be as in Lemma \ref{le:nearness}; it is possible to choose
the $B_i$ to be non-overlapping. Let $\zeta_i:B_i \rightarrow
\mathbb{D}$, $i=1,\ldots,n$ be biholomorphisms. From Remark \ref{re:Tdomain} there is a
corresponding chart $T:\Oqc(\riem_*^P) \rightarrow \Oqc^n$ on the
open subset $V=\{ \phi \in \Oqc(\riem_*^P) \st  \phi_i(\overline{\mathbb{D}}) \subset B_i \}$ containing
$\Lambda([\riem^B,h_0,\riem^B_0])$, given by
$T((\phi_1,\ldots,\phi_n))=(\zeta_1 \circ \phi_1, \ldots, \zeta_n
\circ \phi_n)$.

\begin{theorem}
\label{th:Lambdaholo} There exists an open neighborhood $A_0$ of
$[\riem^B,h_0,\riem^B_{*,0}] \in A_B$ such that  $\Lambda: A_0 \to
\Oqc(\riem^P_*)$ is holomorphic.
\end{theorem}

\begin{proof}  Choose $A_0$ as in Lemma \ref{le:nearness} and the preceding paragraph. 

Using Theorems \ref{th:holo_equivalents} and \ref{th:weak_holo} it is enough to show weak G\^ateaux-holomorphy and local boundedness.  As  before we temporarily drop
the subscript $i$ for ease of notation.

Let $\phi_t$ be as in Lemma \ref{lemma_phit_holo} where it was proved that  $t \mapsto \phi_t(z)$
is holomorphic in $t$. 
The complex structure on $\Oqc(\riem^P_*)$ is defined by
$$
\Oqc(\riem^P_*) \stackrel{T}{\longrightarrow} \Oqc \stackrel{\chi}{\longrightarrow} A^1_{\infty} \oplus \mathbb{C}.
$$
Recall from Section \ref{se:CstructureOqc} that $\chi(f) = (\mathcal{A}(f), f'(0))$ where $\mathcal{A}(f) = f''/f'$. Let $\psi_t = T(\phi_t) = \zeta \circ \phi_t$.  
We now need to prove that $t \mapsto \chi(T(\phi_t))$ satisfies condition (1) of Theorem \ref{th:weak_holo}.

This is immediate for the second component of $\chi$ because $\phi_t(0)$ is holomorphic in $t$.
Since $\phi_t(z)$ and its derivatives are holomorphic in $t$ for fixed $z$, we see that
$$
t \mapsto (\mathcal{A}(\psi_t))(z)
$$
is also holomorphic in $t$. Define  $E_z : A^1_{\infty} \to \mathbb{C}$ by
$E_z(f) = f(z)$. These point evaluation maps are continuous linear
functionals for $z \in \mathbb{D}$. If $\Omega$ is an open subset of
$\disk$ then $\{E_z \st z \in \Omega\}$ is a separating set of
functionals because the holomorphic functions are
determined by their values on an open set.

By Theorem \ref{th:weak_holo}  it remains to show that $\Lambda$ is
locally bounded. From the definition of the complex structure on
$\Oqc(\riem^P_*)$ we must prove that
 $\chi \circ T \circ \Lambda$ is locally bounded.

 It suffices to show that each component of $\chi \circ T \circ \Lambda$ in
 $A^1_{\infty} \oplus \mathbb{C} $ is bounded uniformly on $A_0$.  For any element of
 $A_0$ the corresponding map
$$ 
\psi = \zeta \circ \sigma \circ f^\mu \circ \alpha \circ \tilde{\tau} \in \Oqc
$$
is a holomorphic map from $\mathbb{D}$ to $\mathbb{D}$ satisfying
 $\psi(0)=0$.  By the Schwarz lemma
 \[ |(\psi'(0)| \leq 1.  \]
 Furthermore, by an elementary estimate for univalent maps
 of the disk (using again the above bound on the first derivative)
 \[ |(1-|z|^2)\mathcal{A}(\psi)(z)
  -2\bar{z}| \leq 4 \]
 so
 \[  ||\mathcal{A}(\psi)||_{1,\infty} \leq 6.  \]
\end{proof}

By Theorem \ref{th:LambdaisLinv} and the preceding theorem we now conclude that $\Lambda \circ \PDB^{-1}$ is a local holomorphic inverse of $L$.

\end{subsection}
\end{section}

\begin{section}{Coordinates on the Teichm\"uller space $T(\riem^B)$}

Schiffer variation together with coordinates on the fibers $F_B$ provide local holomorphic coordinate charts for the infinite-dimensional Teichm\"uller space $T(\riem^B)$.

Let $S : \Omega \to T(\riem^P)$ be the coordinates in a neighborhood
of $[\riem^P, f, \riem^P_*] \subset T(\riem^P)$ obtained by Schiffer
variation as described in Theorem \ref{SchifferTheorem}. Let $\riem^P_{\epsilon} = (\riem^P_*)^{\epsilon}$ and recall that $f^{\epsilon} = \nu^{\epsilon} \circ f$.

Note in particular that for any $\epsilon \neq 0$ in $\Omega$, $S(\epsilon)
\neq [\riem^P, f, \riem^P_*]$. This implies that Schiffer variation
in $\riem^B$ is transverse to the fibers $F_B$.

Let $(T,V)$ be a chart on $\Oqc(\riem^P_*)$ as in Section
\ref{se:CstructureOqc}.  Recall that this chart can be chosen so
that for all $\phi \in V$, $\phi(\cdisk)$ is contained in
some fixed open $B \subset \riem^P_*$. Now choose the neighborhoods
$U_i$, $i =1,\ldots,d$, on which the Schiffer variation is to be
performed, to be disjoint from $B$.

Let $f_{\phi}$ be as in Corollary \ref{co:OqcFBbiholo}.
\begin{theorem}
The map
\begin{align*}
(\Omega \times V) & \to T(\riem^B) \\
(\epsilon, \phi) & \mapsto [\riem^B, (f_{\phi})^{\epsilon}, (\riem^P_* \setminus \phi(\disk_0^n))^{\epsilon}]
\end{align*}
is biholomorphic onto its image.
\end{theorem}
\begin{proof}
Because $\PDB$ is a local biholomorphism, it is sufficient
that
\begin{align*}
G:(\Omega \times V) & \to \widetilde{T}^P(\riem^P) \\
(\epsilon, \phi) & \mapsto [\riem^P, \nu^\epsilon \circ
f,\riem^P_\epsilon, \nu^{\epsilon} \circ\phi]
\end{align*}
is a biholomorphism onto its image.
That $G$ is injective follows directly from the definition of $\widetilde{T}^P(\riem^P)$ and the facts: (1) $S : \Omega \to T(\riem^P)$ is injective, and (2) $\riem^P$ has no non-identity automorphisms that are homotopic to the identity since $2g-2+n>0$.  

First, we show that this map is holomorphic.  By Hartog's theorem
is enough to show that the map is separately holomorphic (see
\cite{Mujica} for a version of this theorem in infinite dimensions).
Fixing $\phi$, by Theorem \ref{th:Schiffer_section} the map
$\epsilon \mapsto [\riem^B, f_{\phi}^{\epsilon}, (\riem^P_*
\setminus \phi(\disk_0^n))^{\epsilon}]$ is holomorphic.
Thus since $\mathcal{P}$ is holomorphic, $G$ is holomorphic in
$\epsilon$. Now fix $\epsilon$, and consider the map corresponding
to the second component of $G$:
\begin{align*}
 H:V & \to \widetilde{T}^P(\riem^P) \\
\phi & \mapsto [\riem^P, \nu^\epsilon \circ f,\riem^P_\epsilon,
\nu^{\epsilon} \circ \phi].
\end{align*}
Now $H$ can be written as $H_2 \circ H_1$ where $H_1$ and $H_2$
are given by
\begin{align*}
 H_1:V & \to \Oqc(\riem_\epsilon^P) \\
\phi & \mapsto \nu^\epsilon \circ \phi
\end{align*}
and
\begin{align*}
 H_2: \Oqc(\riem_\epsilon^P) & \to \widetilde{T}^P(\riem^P) \\
\xi & \mapsto [\riem^P, \nu^\epsilon \circ f,\riem^P_\epsilon,
\xi].
\end{align*}
$H_2$ is holomorphic by Theorem \ref{th:Lbiholomorphism}, so it
remains to show that $H_1$ is holomorphic.

Let $\zeta$ be the collection
 of local biholomorphisms of neighborhoods of the punctures on $\riem^P_*$ corresponding to the chart $(T,V)$ on $\Oqc(\riem^P_*)$.  Let $(T_\epsilon,V_\epsilon)$ be the chart on  $\Oqc(\riem_\epsilon^P)$ and $\zeta_\epsilon$ be the corresponding local biholomorphism of $\riem_\epsilon^P$.
 We need to show that $T_\epsilon \circ H_1 \circ T^{-1}$  is holomorphic, i.e. that the map
 \[  \zeta \circ \phi \mapsto \zeta_\epsilon \circ \nu_\epsilon \circ  \phi = (\zeta_\epsilon \circ \nu_\epsilon \circ \zeta^{-1}) \circ (\zeta \circ \phi) \]
 is holomorphic on $\Oqc$.  Composition on the left by a biholomorphism is a local biholomorphism of $\Oqc$ by \cite[Lemma 3.10]{RSOqc}, which establishes the claim.

 Finally we need to show that $G^{-1}$ is holomorphic.  By Theorem \ref{th:Lbiholomorphism} and the fact that Schiffer variation provides a section of the fiber projection, it is clear that the derivative of $G$ is an injective linear map at each point for which $G$ is defined.  Since $G$ is holomorphic and in particular $C^1$, we can apply the inverse function theorem \cite{Lang} to show that $G$ has a $C^1$ inverse.  The derivative of $G^{-1}$ is also complex linear so
  $G^{-1}$ is holomorphic.
\end{proof}

\end{section}

\bibliographystyle{amsplain}

\begin{thebibliography}{99}
 \bibitem{Chae} Chae, S. B.
  {\it Holomorphy and calculus in normed spaces}. With an appendix by Angus
  E. Taylor. Monographs and Textbooks in Pure and Applied Mathematics,
  92. Marcel Dekker, Inc., New York, 1985.
 \bibitem{Gardiner}  Gardiner, F. P. {\it Schiffer's interior
  variation and quasiconformal mapping}.  Duke Math. J. {\bf{42}}
  (1975), 371--380.
\bibitem{Grosse}
   Grosse-{E}rdmann, K.-G. {\it A weak criterion for vector-valued holomorphy}, Math. Proc. Cambridge Philos. Soc., {\bf{169}} (2004), no. 2, 399 --411.
\bibitem{Hubbard}  Hubbard, J. H. {\it Teichm\"uller theory and applications to geometry, topology, and dynamics. Vol. 1.}, Matrix Editions, Ithaca, New York, 2006.
\bibitem{Lang} Lang, S. {\it Differential manifolds}. Second edition.
 Springer-Verlag, New York, 1985.
 \bibitem{Lehto_holomorphicity}  Lehto, O. {\it Quasiconformal
  mappings and singular integrals}.  Instituto Nazionale di Alta
  Matematica., Symposia Mathematica XVIII 429--453, Academic Press
  1976.
 \bibitem{Lehto} Lehto, O.  {\it Univalent functions and
  Teichm\"uller spaces}.  Graduate Texts in Mathematics  {\bf{109}} Springer-Verlag, New York, 1987.
 \bibitem{Lehto-Virtanen} Lehto, O. and Virtanen, K. I.  {\it
 Quasiconformal mappings in the plane}.  2nd edition.  Die
 Grundlehren der mathematischen Wissenshcaften, Band 126.
 Springer-Verlag, New York-Heidelberg, 1973.
 \bibitem{Mujica} Mujica, J. {\it Complex Analysis in Banach Spaces}. North Holland, 1986
 \bibitem{Nag} Nag, S.  {\it The complex analytic theory of
 Teichm\"uller spaces}. Canadian Mathematical Society Series of
  Monographs and Advanced Texts. A Wiley-Interscience Publication.
  John Wiley \& Sons, Inc., New York, 1988.
\bibitem{RadThesis} Radnell, D. {\it Schiffer Varation in {T}eichm{\"u}ller space, Determinant Line Bundles and Modular Functors}, PhD Thesis, Rutgers University, New Brunswick, NJ, 2003.
 \bibitem{RS05}
 Radnell, D.  and Schippers, E.  {\it{Quasisymmetric sewing in rigged
Teichmueller space}}, Commun. Contemp. Math. {\bf 8} (2006) no 4,
481--534.  arXiv:math-ph/0507031.
\bibitem{RSOqc}
 Radnell, D. and Schippers, E. {\it A complex strucure on the set of quasiconformally extendible non-overlapping maps into a Riemann surface}.  to appear in J. Anal. Math.
\bibitem{Segal} Segal, G. \textit{The definition of conformal field theory}, Topology, Geometry
  and Quantum Field Theory (U.~Tillmann, ed.), London Mathematical Society
  Lecture Note Series, vol. 308, Cambridge University Press, 2004, Orginial
  preprint 1988, pp.~421--576.
 \bibitem{Teo} Teo, L.-P.  {\it Velling-Kirillov metric on the
  universal Teichm\"uller curve}.  J. Anal. Math. {\bf 93} (2004),
  271--307.
 \bibitem{TTmem} Takhtajan, L. A. and Teo, L-P.  {\it
  Weil-Petersson metric on the universal Teichm\"uller space}.
  Mem. Amer. Math. Soc., (2006).
\bibitem{Slodkowski} S{\l}odkowski, Z.
{\it Holomorphic motions and polynomial hulls},
Proc. Amer. Math. Soc., {\bf{111}}(2), (1991), 347--355
\end{thebibliography}

\end{document}